\documentclass[11pt,reqno]{amsart}
\usepackage[utf8]{inputenc}

\usepackage{amssymb, amsmath, amsthm}

\usepackage[colorlinks]{hyperref}
\usepackage[nameinlink]{cleveref}
\hypersetup{
  linkcolor=[rgb]{0.3,0.3,0.6},
  citecolor=[rgb]{0.2, 0.6, 0.2},
  urlcolor=[rgb]{0.6, 0.2, 0.2}
}

\usepackage[alphabetic,lite]{amsrefs}
\usepackage{verbatim}
\usepackage{amscd}   
\usepackage[all]{xy} 
\usepackage{youngtab} 
\usepackage{young} 
\usepackage{ytableau}
\usepackage{tikz}
\usepackage{ mathrsfs }
\usepackage{cases}
\usepackage{array}
\usepackage{cellspace}
\usepackage{calligra,mathrsfs}
\usepackage{bm}
\usepackage{graphicx}
\usepackage{rank-2-roots}
\usepackage{float}

\newcommand{\defi}[1]{{\bf\upshape\sffamily #1}}
\DeclareMathOperator{\ShHom}{\mathscr{H}\text{\kern -3pt {\calligra\large om}}\,}

\newcommand{\BP}{{\mathbf B}}
\newcommand{\bw}{\bigwedge}

\def\kk{{\mathbf k}}
\renewcommand{\ll}{\lambda}

\newcommand{\oo}{\otimes}

\newcommand{\NN}{\mathbb{N}}

\newcommand{\Ext}{\operatorname{Ext}}

\newcommand{\GL}{\operatorname{GL}}
\newcommand{\Hom}{\operatorname{Hom}}

\newcommand{\rk}{\operatorname{rank}}

\newcommand{\Sym}{\operatorname{Sym}}

\newcommand{\rank}{\operatorname{rank}}

\newcommand{\coker}{\operatorname{coker}}
\renewcommand{\det}{\operatorname{det}}

\newcommand{\bb}[1]{\mathbb{#1}}

\renewcommand{\rm}[1]{\textrm{#1}}
\newcommand{\mc}[1]{\mathcal{#1}}
\newcommand{\mf}[1]{\mathfrak{#1}}
\newcommand{\ol}[1]{\overline{#1}}
\newcommand{\op}[1]{\operatorname{#1}}

\newcommand{\ul}[1]{\underline{#1}}

\def\PP{{\mathbf P}}

\def\lra{\longrightarrow}

\newcommand{\stacks}[1]{\cite[Tag \href{https://stacks.math.columbia.edu/tag/#1}{#1}]{stacksProject}}

\newtheorem{theorem}{Theorem}[section]
\newtheorem*{theorem*}{Theorem}
\newtheorem*{problem*}{Problem}
\newtheorem{lemma}[theorem]{Lemma}

\newtheorem*{corollary*}{Corollary}

\newtheorem*{main-thm*}{Main Theorem}
\newtheorem*{coh-flag*}{Cohomology Vanishing on Flag Varieties}
\newtheorem*{coh-proj*}{Cohomology Vanishing on Projective Space}

\theoremstyle{definition}

\newtheorem*{definition*}{Definition}

\newtheorem{example}[theorem]{Example}

\theoremstyle{remark}
\newtheorem{remark}[theorem]{Remark}
\newtheorem*{remark*}{Remark}

\numberwithin{equation}{section}

\newcommand{\claudiu}[1]{{\color{red} \sf $\clubsuit\clubsuit\clubsuit$ Claudiu: [#1]}}

\tikzset{
  treenode/.style = {align=center, inner sep=0pt, text centered,solid,thin,
    font=\sffamily},
  arn_n/.style = {treenode, circle, white, font=\sffamily\bfseries, draw=black,
    fill=black, text width=.5em},
  arn_nl/.style = {treenode, circle, white, font=\sffamily\bfseries, draw=black,
    fill=black, text width=1.5em},  
  arn_r/.style = {treenode, circle, red, draw=red, 
    text width=.5em, very thick},
  arn_v/.style = {treenode, circle, black, font=\sffamily\bfseries, draw=black, text width=1.2em},
  arn_x/.style = {treenode, rectangle, draw=black,
    minimum width=.5em, minimum height=0.5em},
  dott/.style={edge from parent/.style={dotted, very thick,circle,draw}},
  emph/.style={edge from parent/.style={dashed, very thick,circle,draw}},
  norm/.style={edge from parent/.style={solid,thin,circle,draw}}
}
\makeatletter
\def\labelbox#1{%
  \hbox{%
    \setbox\z@=\hbox{$\m@th\labelstyle{\,#1\,}$}%
    \setbox\tw@=\hbox{$\m@th\labelstyle\,$}%
    \dimen@=\ht\z@ \advance\dimen@ by \wd\tw@ \ht\z@=\dimen@
    \dimen@=\dp\z@ \advance\dimen@ by \wd\tw@ \dp\z@=\dimen@
    \box\z@
  }%
}
\makeatother
\begin{document}

\title{Diagonal F-thresholds for determinants and Pfaffians}

\author{Barbara Betti}
\address{Otto-von-Guericke-University Magdeburg, Universitätsplatz 2, Magdeburg, Germany}
\email{barbara.betti@ovgu.de}

\author{Claudiu Raicu}
\address{Department of Mathematics, University of Notre Dame, 255 Hurley, Notre Dame, IN 46556\newline
\indent Institute of Mathematics ``Simion Stoilow'' of the Romanian Academy}
\email{craicu@nd.edu}

\author{Francesco Romeo}
\address{University of Cassino and Southern Lazio, 
DIEI, Department of Electrical and Information \newline \indent Engineering}
\email{francesco.romeo@unicas.it}

\author{Jyoti Singh}
\address{Department of Mathematics, Visvesvaraya National Institute of Technology, Nagpur, India}
\email{jyotisingh@mth.vnit.ac.in}


\date{\today}


\begin{abstract} 
We compute the diagonal $F$-thresholds of determinantal hypersurfaces arising from a generic matrix and from a generic symmetric matrix, as well as of the Pfaffian hypersurface arising from a generic skew-symmetric matrix of even size. The main ingredient is a cohomology vanishing theorem for certain line bundles on flag varieties in characteristic $p$. In the cases of the generic matrix and the generic skew-symmetric matrix, we show that the diagonal $F$-threshold attains its minimal possible value, namely the negative of the $a$-invariant. The symmetric case is more subtle and relies in addition on a polynomiality result for representations afforded by cohomology, building on work of the second author with VandeBogert.
\end{abstract}

\maketitle

\section{Introduction}
\label{sec:intro}

In characteristic zero, the representation theory of $\GL_n$ admits two remarkable incarnations: geometrically, through the Borel--Weil--Bott theorem \cites{borel-weil,bott}, which realizes irreducible representations in the cohomology of line bundles on flag varieties, and algebraically, through standard monomial theory and the commutative algebra of determinantal varieties, as developed classically by De Concini--Eisenbud--Procesi \cite{DCEP}. In characteristic $p>0$, both the structure of cohomology and that of determinantal rings are considerably less understood, reflecting the subtle ways in which the Frobenius endomorphism interacts with taking determinants. The goal of this paper is to begin making this interaction precise, by characterizing the diagonal $F$-thresholds of hypersurface rings defined by determinants and Pfaffians. For a Noetherian ring $R$ of positive characteristic, \defi{$F$-thresholds} are important invariants that measure the asymptotic containment relations between regular and Frobenius powers for ideals in $R$. These invariants grew out of work of Huneke, Musta\c t\u a, Takagi, Watanabe \cites{MTW,HMTW} in the regular setting, and were later proved to exist in full generality in \cite{DSNP}. However, they remain difficult to compute in many important contexts, and we illustrate how this can be done for determinantal rings by establishing new vanishing results for cohomology on flag varieties, and by employing fundamental results in modular representation theory and the theory of polynomial functors.

We work over an algebraically closed field $\kk$ of characteristic $p>0$ and we denote by $R$ a standard graded $\kk$-algebra with maximal homogeneous ideal $\mf{m}$. We consider the \defi{Frobenius powers} $\mf{m}^{[q]} = \langle f^{q} : f\in \mf{m}\rangle$, where $q=p^s$. The Frobenius powers provide a way of understanding the action of the Frobenius endomorphism on $R$, their relationship with regular powers of $\mf{m}$ is controlled by
\[ v_R(q) = \max\{d \in \NN \ | \ \mf{m}^d \not\subseteq \mf{m}^{[q]}\}\]
We can view $v_R(q)$ as the \defi{socle degree} of the finite length algebra $R/\mf{m}^{[q]}$, that is the largest degree in which a graded component of $R/\mf{m}^{[q]}$ is non-zero. To simplify the notation, we write $c(R)$ instead of the more common $c^{\mf{m}}(\mf{m})$ for the \defi{(diagonal) $F$-threshold of $R$}, which is defined as the limit
\begin{equation}\label{eq:def-cR}
 c(R) = \lim_{q\to\infty}\frac{v_R(q)}{q}.
\end{equation}

\begin{main-thm*}
 Let $S=\kk[x_{ij}]$ a polynomial ring in the entries of an $n\times n$ matrix. The diagonal $F$-threshold for the hypersurface ring $R=S/\langle f\rangle$ is computed as follows.
 \begin{enumerate}
  \item If $(x_{ij})$ is a generic matrix and $f=\det(x_{ij})$ then $c(R)=n^2-n$.
  \item If $(x_{ij})$ is a generic symmetric matrix and $f=\det(x_{ij})$ then $c(R)=(n^2-1)/2$.
  \item If $(x_{ij})$ is a generic skew-symmetric matrix of even size and $f=\op{Pf}(x_{ij})$ then $c(R)=(n^2-2n)/2$.
 \end{enumerate} 
\end{main-thm*}

Case (1) of the Main Theorem confirms \cite{BMRS}*{Conjecture~29}, which has been verified for $n\leq 4$ in \cite{BMRS}*{Corollary~28}. As remarked in \cite{BMRS}*{Theorem~6}, the formula for $F$-thresholds can be extended to non-square matrices, and the argument we present in Section~\ref{sec:Fthresh-detl} applies in this greater generality. In cases (1) and (3) of the Main Theorem, the formula for $c(R)$ agrees with $-a(R)$, the negative of the $a$-invariant of~$R$ (see Section~\ref{subsec:hyper} for the definition), which is shown in \cite{DS-NB}*{Theorem~4.9} to always give a lower bound for the diagonal $F$-threshold; the difficulty then lies in bounding $c(R)$ from above. By contrast, for the symmetric determinant we have $-a(R)=(n^2-n)/2$, so the lower bound is no longer optimal. Our approach bounds $c(R)$ from both above and below, as explained in Section~\ref{sec:Fthresh-symm-det}. For the Pfaffian hypersurface, our estimate for the upper bound of $c(R)$ reduces via a degeneration argument to the corresponding statement for the determinant of a generic matrix of half the size, and it is explained in Section~\ref{sec:Fthresh-Pfaff}.

Our upper bounds on $c(R)$ for the generic determinant and the generic symmetric determinant are obtained by giving a cohomological interpretation of the graded pieces of $R$, and by characterizing $v_q(R)$ via cohomology vanishing. For each degree $d$ we construct a (non-exact) Koszul complex $\mc{K}^{\bullet}$, whose terms have no higher cohomology, and whose zeroth hypercohomology $\bb{H}^0(\mc{K}^{\bullet})$ measures $\left(R/\mf{m}^{[q]}\right)_d$. The homology sheaves of $\mc{K}^{\bullet}$ can be described explicitly starting from tautological sheaves, using standard multilinear-algebra constructions and Frobenius twists. The remaining challenge is to determine the range in which the cohomology of each homology sheaf vanishes, and the key tool for addressing this is introduced next.

We consider the \defi{(complete) flag variety} $Fl_n$ parametrizing complete flags of subspaces
\[
 V_{\bullet}:\qquad 0\subset V_1\subset\cdots\subset V_{n-1} \subset \kk^n,\qquad\text{ where }\dim(V_i)=i.
\]
The line bundles on $Fl_n$ are parametrized by tuples (weights) $\ul{y}=(y_1,\cdots,y_n)\in\bb{Z}^n$, and we write $\mc{O}_{Fl_n}(\ul{y})$ for the line bundle corresponding to $\ul{y}$ (see also Section~\ref{subsec:flag-theory}). The use of cohomological techniques on Grassmannians and flag varieties to study determinantal varieties goes back to the work of Lascoux on syzygies \cite{lascoux}, and is described extensively in \cite{weyman}. Most of the results in these works require characteristic zero, as they rely on the Borel--Weil--Bott theorem. Nevertheless, the general framework can be formulated over an arbitrary field $\kk$; to obtain results in characteristic $p$, one must find suitable modifications of the characteristic zero Borel--Weil--Bott statements (see \cite{BCRV}*{Sections~9.5,~9.6, Chapter~10} for a concrete illustration). The following is the main vanishing statement that will be needed to estimate $F$-thresholds.

\begin{coh-flag*}  Suppose that $\op{char}(\kk)=p>0$ and let $q=p^s$ for some $s\geq 0$. If $\ll=(\ll_1,\dots,\ll_{n-1})$ is a partition of size $|\ll|\leq (n-1-j)q-1$ and $e\geq(1+j)q$ then
 \begin{equation}\label{eq:vanishing-flag-varieties}
 H^k\left(Fl_n,\mc{O}_{Fl_n}(\ll_1,\dots,\ll_{n-1},e)\right) = 0 \quad\text{ for }k\leq j.
 \end{equation}
\end{coh-flag*}

In the case $j=0$, the hypotheses above imply that $e>\ll_{n-1}$, so the weight $(\ll_1,\cdots,\ll_{n-1},e)$ is not dominant. The vanishing of $H^0$ is then a folklore generalization of the Borel--Weil theorem to arbitrary characteristic (see \cite{BCRV}*{Theorem~9.8.5}). The case $j=1$ follows from a classical result of Andersen characterizing the (non-)vanishing of $H^1$ \cite{andersen}*{Theorem~3.6}. The general case is explained in Section~\ref{sec:coh-vanishing}.

\begin{example}\label{ex:coh-vanishing-Flag-optimal} To see the cohomology vanishing statement (\ref{eq:vanishing-flag-varieties}) is optimal, consider the line bundle 
\[\mc{L} = \mc{O}_{Fl_n}(1^{n-1-j},0^j,1+j).\] 
It has a single non-vanishing cohomology group, namely $H^j(Fl_n,\mc{L})=\bw^n\kk^n\simeq\kk$, which shows the optimality of our result when $q=1$. Using the fact that $Fl_n$ is globally $F$-split, we get an inclusion (see also \cite{and-Frob}*{Proposition~3.3})
\[ H^j(Fl_n,\mc{L}) \hookrightarrow H^j(Fl_n,\mc{L}^q)\text{ for all }q=p^s.\]
Since $\mc{L}^q = \mc{O}_{Fl_n}(\ll_1,\cdots,\ll_{n-1},e)$ where $e=(1+j)q$ and $|\ll| = (n-1-j)q$, we conclude that the hypothesis $|\ll|\leq (n-1-j)q-1$ is optimal to obtain vanishing.
\end{example}

Cohomology results for line bundles on full flag varieties can often be related to those for higher-rank vector bundles on partial flag varieties such as Grassmannians or the projective space (see \cite{BCRV}*{Chapter~9} and \cite{rai-vdb1} for some examples). Let $\PP$ denote the projective space parametrizing 1-dimensional subspaces of~$\kk^n$, then $\PP$ comes equipped with a tautological short exact sequence
\begin{equation}\label{eq:taut-ses-PP}
 0 \lra \mc{O}_{\PP}(-1) \lra \mc{O}_{\PP}^{\oplus n} \lra \mc{Q} \lra 0
\end{equation}
where $\mc{Q}$ is the tautological rank $(n-1)$ quotient bundle. If we write $\bb{S}_{\ll}$ for the Schur functor associated to a partition $\ll$, then the cohomology of the line bundle $\mc{O}_{Fl_n}(\ll_1,\cdots,\ll_{n-1},e)$ coincides with that of the vector bundle $\bb{S}_{\ll}\mc{Q} \oo \mc{O}_{\PP}(-e)$ on $\PP$. We can then view the following result as a generalization of our earlier cohomology vanishing on flag varieties, which we also verify in Section~\ref{sec:coh-vanishing}.

\begin{coh-proj*}  Suppose that $\op{char}(\kk)=p>0$ and let $q=p^s$, $s\geq 0$. If $\mc{P}$ is a polynomial functor of degree at most $(n-1-j)q-1$ and if $e\geq(1+j)q$ then
 \begin{equation}\label{eq:vanishing-Polyfun}
  H^k\left(\PP,\mc{P}(\mc{Q})\oo \mc{O}_{\PP}(-e)\right) = 0 \quad\text{ for }k\leq j.
 \end{equation}
\end{coh-proj*} 

The result is vacuous for $j\geq n-1$, since the hypothesis forces $\mc{P}$ to have negative degree; however, it is already interesting for $j=n-2$. To explain this case, we let $\Omega$ denote the cotangent sheaf on $\PP$, and use the fact that $\bw^{n-1}\mc{Q} \simeq \mc{O}_{\PP}(1)$ along with the identifications
\[\Omega = \mc{Q}^{\vee} \oo \mc{O}_{\PP}(-1),\quad\text{and}\quad \mc{Q}  =  \left(\bw^{n-2}\mc{Q}^{\vee}\right) \oo \bw^{n-1}\mc{Q} = \bw^{n-2}\Omega \oo \mc{O}_{\PP}(n-1).\]
If we let $d$ denote the degree of $\mc{P}$ and let $\tilde{\mc{P}}=\mc{P} \circ \bw^{n-2}$, we obtain 
\[ \mc{P}(\mc{Q})\oo \mc{O}_{\PP}(-e) = \tilde{\mc{P}}(\Omega)\oo \mc{O}_{\PP}((n-1)d-e),\]
If we take $j=n-2$ then we have $d<q$ and $e\geq(n-1)q$, hence $(n-1)d-e<0$, and the desired vanishing \eqref{eq:vanishing-Polyfun} occurs as a special case of \cite{rai-vdb1}*{Theorem~4.1(3)}. The general case requires more work, and we prove it using an inductive argument, along with the vanishing result on flag varieties.

\medskip

\noindent{\bf Organization.} Section~\ref{sec:prelim} collects the necessary background on $F$-thresholds and $a$-invariants, cohomology, polynomial functors, and modular representation theory of $\GL_n$. In Section~\ref{sec:coh-vanishing} we present the main vanishing results for cohomology, as well as the polynomiality result for cohomology needed to bound $c(R)$ from below in the symmetric case. Sections~\ref{sec:Fthresh-detl},~\ref{sec:Fthresh-symm-det},~\ref{sec:Fthresh-Pfaff} compute $c(R)$ in cases (1), (2), (3) of the Main Theorem, respectively.

\section{Preliminaries}
\label{sec:prelim}

We work throughout over an algebraically closed field $\kk$ of characteristic $p>0$. We will reserve the letter $q$ to denote a power $q=p^s$ of the characteristic of $\kk$. We often refer to elements $\ul{y}\in\bb{Z}^n$ as \defi{weights}, and say that $\ul{y}$ is \defi{dominant} if it belongs to the subset $\bb{Z}^n_{dom}$ where $y_1\geq\cdots\geq y_n$. If moreover $y_n\geq 0$ then we will call $\ul{y}$ a \defi{partition}. The \defi{size} of a weight $\ul{y}$ is $|\ul{y}|=y_1+\cdots+y_n$.

\subsection{Hypersurface rings and duality \cite{bru-her}}
\label{subsec:hyper}

Let $S=\kk[x_1,\dots,x_r]$, let $0\neq f\in S$ denote a homogeneous polynomial of degree $k>0$, and set $R=S/\langle f\rangle$. We abuse notation and write $\mf{m}$ for both the maximal homogeneous ideal of $S$ and that of $R$. 
For $q=p^s$ we set $\ol{S}=S/\mf{m}^{[q]}$ and $\ol{R}=R/\mf{m}^{[q]}$. Given a graded module $M$ we set the initial degree and the end degree of $M$ to be:
\[\op{indeg}(M) = \inf\{i : M_i\neq 0\},\qquad \op{endeg}(M) = \sup\{i : M_i\neq 0\}.\]
Writing $\omega_A$ for the canonical module of a Cohen--Macaulay graded algebra $A$, we have 
\[a(A) = \op{indeg}(\omega_A),\]
where $a(A)$ is the $a$-invariant of $A$. Furthermore, by Theorem 3.3.7 in \cite{bru-her} we obtain
\[\omega_S = S(-r)\text{ and }\omega_R = \Ext^1_S(R,\omega_S) = R(k-r),\]
and it follows that $-a(R)=r-k$ gives a lower bound for the diagonal $F$-threshold 
\begin{equation}\label{eq:cR>aR-hypersurface}
    c(R) \geq r-k.
\end{equation}
This bound will turn out to be sharp for the generic determinant and Pfaffian, but not in the case of the symmetric determinant. For the latter we use a different method to bound from below
\[v_R(q) = \op{endeg}(\ol{R}) = -\op{indeg}(\omega_{\ol{R}}).\]
Using that $\ol{S}$ is Artinian Gorenstein with socle in degree $(q-1)r$, and $\ol{R}=\ol{S}/\langle f\rangle$, it follows that
\[\omega_{\ol{S}} = \Hom_{\kk}(\ol{S},\kk) = \ol{S}((q-1)r)\quad\text{and}\quad \omega_{\ol{R}} = \Hom_{\ol{S}}(\ol{R},\omega_{\ol{S}}) = (0:_{\ol{S}}f)((q-1)r).\]
If we let $t=\op{indeg}(0:_{\ol{S}}f)$ then $\op{indeg}(\omega_{\ol{R}})=t-(q-1)r$, and therefore
\begin{equation}\label{eq:vRq-from-indeg0f}
    v_R(q) = (q-1)r - \op{indeg}(0:_{\ol{S}}f).
\end{equation}
Finding optimal lower bounds for $v_R(q)$ is then equivalent to find minimal degree annihilators of $f$ in $\ol{S}$, which we determine for the symmetric determinant in characteristic $2$ in Section~\ref{subsec:char2-symdet}.

\subsection{Representation theory basics \cite{green,jantzen}}
\label{subsec:rep-thy}

Given a group $G$ and a finite dimensional $G$-representation $W$, we say that $W$ is \defi{simple} (or \defi{irreducible}) if it contains no proper subrepresentation $0\subsetneq W'\subsetneq W$. Every finite dimensional representation $W$ admits a (Jordan--H\"older) filtration
\begin{equation}\label{eq:filtr-Wbullet}
0=W_0\subset W_1\subset\cdots\subset W_r = W
\end{equation}
such that the quotients (composition factors) $W_{i+1}/W_i$ are simple. Although the filtration itself is not unique, the simple composition factors, counted with multiplicity and considered up to isomorphism (and reordering), are uniquely determined by $W$. Given a simple $G$-representation $L$, we write $[W:L]$ for the \defi{multiplicity} of $L$ in $W$, that is, the number of quotients $W_{i+1}/W_i$ isomorphic to $L$ in a filtration \eqref{eq:filtr-Wbullet}. We say that $L$ \defi{occurs} in (or is composition factor of) $W$ if $[W:L]>0$.

We next consider the algebraic group $G=\GL_n(\kk)$ and recall some properties of finite dimensional $G$-representations. The simple $G$-representations are indexed by the collection of \defi{dominant weights}
\[\bb{Z}^n_{dom} = \{\ll\in\bb{Z}^n : \ll_1\geq\cdots\geq\ll_n\}.\]
We write $L_{\ll}$ for the simple representation corresponding to $\ll$. Let $\omega_i$ denote the $i$-th \defi{fundamental weight} $(1^i,0^{n-i})=(1,\dots,1,0,\dots,0)$, where $1$ and $0$ appear, respectively, $i$ and $(n-i)$ times, then
\[ L_{\omega_i} = \bw^i\kk^n\quad\text{for }i=1,\dots,n.\]
Every dominant weight $\ll$ is expressible uniquely as
\begin{equation}\label{eq:lam-from-fundamentals}
\ll = a_1\omega_1+\cdots+a_n\omega_n,\quad\text{where }a_1,\dots,a_{n-1}\in\bb{Z}^n_{\geq 0},\ a_n\in\bb{Z}.
\end{equation}
More precisely, we have $a_i=\ll_i-\ll_{i+1}$ for $i=1,\dots,n-1$, and $a_n=\ll_n$. We refer to $L_{\omega_1}=\kk^n$ as the \defi{standard representation}, and to $L_{\omega_n}=\det(\kk^n)$ as the \defi{determinant representation}. We have that $L_{\omega_n}$ is $1$-dimensional, and it has the property that
\[ L_{\ll} \oo L_{\omega_n} = L_{\ll+\omega_n} \quad\text{ for all }\ll\in\bb{Z}^n_{dom}.\]
This implies that
\[ L_{-\omega_n} = \Hom_{\kk}\left(\bw^n\kk^n,\kk\right) = L_{\omega_n}^{\vee},\quad\text{and}\quad L_{-a\omega_n} = \left(L_{\omega_n}^{\vee}\right)^{\oo a}\text{ for }a\geq 0.\]
We say that $L_{\ll}$ is a \defi{polynomial representation} if $\ll_n\geq 0$ (that is, if $\ll$ is a partition). More generally, a $G$-representation $W$ is polynomial if all of its composition factors $L_{\ll}$ are polynomial. 

If $p=\op{char}(\kk)$ then (iterates of) Frobenius give rise to natural operations on the category of representations. If we write $\phi^{(i)}:G\lra G$ for the endomorphism sending $\phi^{(i)}(g) = g^{p^i}$ then for every $G$-representation $W$ we get via restriction along $\phi^{(i)}$ a new (twisted) representation, denoted $W^{(i)}$. In general we have $W^{(0)}=W$, $(W^{(i)})^{(j)} = W^{(i+j)}$, and for simple representations we have $L_{\ll}^{(i)} = L_{p^i\cdot \ll}$. We note that twisting is an exact operation compatible with tensor products, and that it preserves the subcategory of polynomial representations. 

We say that a weight $\ll$ as in \eqref{eq:lam-from-fundamentals} is \defi{$p$-restricted} if $0\leq a_i<p$ for all $i$ (in particular $\ll$ is a partition and $L_{\ll}$ is polynomial). Using the $p$-adic expansions $a_i = \sum_j a^{(j)}_i\cdot p^j$, $0\leq a^{(j)}_i<p$, we can define $\ll^j = \sum_i a^{(j)}_i\cdot\omega_i$ and express every partition $\ll$ uniquely as 
\begin{equation}\label{eq:padic-ll} 
\ll = \ll^0 + p\ll^1 + \cdots + p^s \ll^s,\quad\text{ where }\ll^i\text{ is $p$-restricted for all $i$}.
\end{equation}
By the Steinberg Tensor Product theorem \cite{jantzen}*{Corollary~II.3.17}, the decomposition \eqref{eq:padic-ll} gives rise to an isomorphism
\begin{equation}\label{eq:STP} 
L_{\ll} = L_{\ll^0} \oo L_{\ll^1}^{(1)} \oo \cdots \oo L_{\ll^s}^{(s)}.
\end{equation}

We define the \defi{degree} of $L_{\ll}$ to be $|\ll|=\ll_1+\cdots+\ll_n$, and say that a representation $W$ is homogeneous of degree $d$ if all of its composition factors $L_{\ll}$ have $|\ll|=d$.  Since the degree is additive relative to tensor products of representations, it follows that if $L_{\ll}$ occurs in a tensor product $L_{\mu}\oo L_{\nu}$ then $|\ll|=|\mu|+|\nu|$. Combining this with \eqref{eq:STP}, we get the following consequence that will be needed in Section~\ref{subsec:lowbd-symdet-coh}.

\begin{lemma}\label{lem:pres-in-tensor}
    Suppose that $\mu$ is $p$-restricted, and $\ll=\ll^0+p\ll^1+\cdots+p^s\ll^s$, where each $\ll^i$ is $p$-restricted. If $L_{\ll}$ occurs in a tensor product $L_{p^s\mu}\oo L_{\nu}$ for some partition $\nu$ then $|\ll^s|\geq |\mu|$. 
\end{lemma}

\begin{proof}
    Writing $\nu = \nu^{0}+p\nu^1+\cdots+p^s\nu^s$, where $\nu^0,\dots,\nu^{s-1}$ are $p$-restricted and $\nu^s$ is a partition,~then
    \[ L_{p^s\mu}\oo L_{\nu} = L_{\nu^0}\oo L_{\nu^1}^{(1)} \oo \cdots L_{\nu^{s-1}}^{(s-1)} \oo \left(L_{\mu} \oo L_{\nu^{s}} \right)^{(s)}.\]
    It follows that if $L_{\ll}$ is a composition factor of $L_{p^s\mu}\oo L_{\nu}$ then $\ll^i=\nu^i$ for $i<s$ and $L_{\ll^s}$ is a composition factor of $L_{\mu} \oo L_{\nu^{s}}$, which implies $|\ll^s|=|\mu|+|\nu^s|\geq |\mu|$, as desired.
\end{proof}

\subsection{Polynomial functors \cite{green,fri-sus}}
\label{subsec:poly-fun}

We consider the category $\mf{Pol}_d$ of \defi{(strict) polynomial functors} of degree $d$ over the field $\kk$ (see \cite{rai-vdb1}*{Section~2} for a quick summary of the relevant background). The main examples of polynomial functors are associated to a partition $\ll$ of $d$, and are given by the \defi{Schur functors} $\bb{S}_{\ll}$ and the \defi{Weyl functors} $\bb{W}_{\ll}$. When $\ll=(d)$ we have $\bb{S}_{\ll}=\Sym^d$ is the usual $d$-th symmetric power functor and $\bb{W}_{\ll}=\op{D}^d$ is the $d$-th divided power functor. When $\ll=(1^d)$, one has that $\bb{S}_{\ll}=\bb{W}_{\ll}$ coincide with the exterior power functor $\bw^d$. $\mf{Pol}_d$ is a finite length category with simple objects $\bb{L}_{\ll}$ parametrized by partitions of $d$. We have that $\bb{L}_{\ll}$ can be realized as the unique simple subfunctor of $\bb{S}_{\ll}$, and the unique simple quotient of~$\bb{W}_{\ll}$. All other composition factors of $\bb{S}_{\ll}$ and $\bb{W}_{\ll}$ have the form $L_{\mu}$, where $\mu<\ll$ in the dominance order.

The \defi{Frobenius $q$-power functor $F^q$} arises as the special case $F^q = \bb{L}_{\ll}$ for the singleton partition~$\ll=(q)$ and hence it is a subfunctor of $\Sym^q$. We define the \defi{$q$-truncated symmetric powers}~as
\begin{equation}\label{eq:def-TqSymd}
T_q\Sym^d  = \coker\left(F^q \oo \Sym^{d-q} \overset{\phi}{\lra} \Sym^d \right),
\end{equation}
where the natural transformation $\phi$ is induced by the inclusion $F^q\hookrightarrow \Sym^q$ followed by the multiplication map on symmetric powers.

Given a polynomial functor $\mc{P}\in \mf{Pol}_d$, evaluation on $\kk^n$ yields a polynomial representation $\mc{P}(\kk^n)$ of $\GL_n$ of degree $d$. The most familiar examples of degree $d$ representations are the symmetric powers $\Sym^d\kk^n$, exterior powers $\bw^d\kk^n=L_{\omega_d}$, tensor powers $T^d\kk^n=\left(\kk^n\right)^{\oo d}$ and the divided powers $D^d\kk^n$. In general one has $\bb{L}_{\ll}(\kk^n)=L_{\ll}$ if $\ll_{n+1}=0$ (that is, if $\ll$ has at most $n$ parts), and $\bb{L}_{\ll}(\kk^n)=\bb{S}_{\ll}(\kk^n)=\bb{W}_{\ll}(\kk^n)=0$ if $\ll_{n+1}>0$. Note that if $n\geq d$ then every partition $\ll$ of $d$ satisfies $\ll_{n+1}=0$. In particular, the assignment $\mc{P}\mapsto \mc{P}(\kk^n)$ gives an equivalence of categories between $\mf{Pol}_d$ and the category of polynomial representations of $\GL_n$ of degree $d$.

\subsection{Flag varieties and flag bundles \cite{BCRV}}
\label{subsec:flag-theory}

In this section we recall some of the formalism and notation from \cite{BCRV}*{Section~9}, which will be useful for our arguments in Section~\ref{sec:coh-vanishing}. We let $\BP$ be an algebraic variety, and let $\mc{E}$ be a locally free sheaf of finite rank on $\BP$. We write $\bb{F}_{\BP}(\mc{E})$ for the \defi{relative flag bundle}, with structure morphism $\pi:\bb{F}_{\BP}(\mc{E})\lra\BP$, and tautological line bundles
\[ \mc{O}_{\bb{F}_{\BP}(\mc{E})}(\ul{y}),\quad\ul{y}\in\bb{Z}^n.\]
Using these conventions, if the indexing weight $\ul{y}$ is dominant, then we write the direct image as
\[ \pi_*\left(\mc{O}_{\bb{F}_{\BP}(\mc{E})}(\ul{y})\right) = \bb{S}_{\ul{y}}\mc{E},\quad R^j\pi_*\left(\mc{O}_{\bb{F}_{\BP}(\mc{E})}(\ul{y})\right) = 0\text{ for }j\neq 0.\]
We consider the projective bundle $\PP=\bb{P}_{\BP}(\mc{E})=\ul{\op{Proj}}(\Sym_{\mc{O}_{\BP}}\mc{E})$, with tautological exact sequence
\[ 0 \lra \mc{R} \lra \pi^*\mc{E} \lra \mc{Q} \lra 0,\]
where $\mc{Q} = \mc{O}_{\PP}(1)$. Following the discussion in \cite{BCRV}*{Section~9.6}, we can factor $\pi$ as
\begin{equation}\label{eq:factorization-pi-fg}
 \xymatrix{
 {\bb{F}_{\BP}(\mc{E})} \ar[r]_-f \ar@/^1.5pc/[rr]^-\pi & \PP \ar[r]_-g & \BP \\
 }
\end{equation}
and think of $f$ as the structure map of the flag bundle $\bb{F}_{\PP}(\mc{R})$, thus identifying $\bb{F}_{\BP}(\mc{E})=\bb{F}_{\PP}(\mc{R})$. We then get an identification
\[ \mc{O}_{\bb{F}_{\BP}(\mc{E})}(y_1,y_2,\dots,y_n) = f^*(\mc{Q}^{y_1}) \oo \mc{O}_{\bb{F}_{\PP}(\mc{R})}(y_2,\dots,y_n)\]
and the projection formula \stacks{01E8} implies
\[R^qf_*\left(\mc{O}_{\bb{F}_{\BP}(\mc{E})}(y_1,y_2,\dots,y_n)\right) = \mc{Q}^{y_1} \oo R^qf_*\left(\mc{O}_{\bb{F}_{\PP}(\mc{R})}(y_2,\dots,y_n)\right)\text{ for all }q.\]
We can combine this identification with the relative Leray spectral sequence \stacks{01F6}
\[E_2^{p,q}=R^pg_*\left(R^qf_*\left(\mc{O}_{\bb{F}_{\BP}(\mc{E})}(y_1,y_2,\dots,y_n)\right)\right) \Longrightarrow R^{p+q}\pi_*\left(\mc{O}_{\bb{F}_{\BP}(\mc{E})}(y_1,y_2,\dots,y_n)\right)\]
to obtain the implication
\begin{equation}\label{eq:inductive-vanishing}
 \text{if }R^k f_*\left(\mc{O}_{\bb{F}_{\PP}(\mc{R})}(y_2,\dots,y_n)\right) = 0\text{ for }k\leq j\text{ then }R^k \pi_*\left(\mc{O}_{\bb{F}_{\BP}(\mc{E})}(y_1,y_2,\dots,y_n)\right) = 0\text{ for }k\leq j.
\end{equation}


\section{Cohomology vanishing and polynomiality}
\label{sec:coh-vanishing}

The goal of this section is to verify the main vanishing and polynomiality results for cohomology on flag varieties and projective spaces. For the vanishing results, the proof structure makes it more convenient to formulate the arguments in the relative setting using the setup from Section~\ref{subsec:flag-theory} -- we explain this in Section~\ref{subsec:coh-van-flagbundles}. We prove the polynomiality result in Section~\ref{subsec:coh-polynomial-rep}, following closely the strategy from \cite{rai-vdb1}*{Section~3}.

\subsection{Cohomology vanishing on flag bundles}
\label{subsec:coh-van-flagbundles}

We begin by formulating our cohomology vanishing on flag varieties in the more general context of flag bundles.

\begin{theorem}\label{thm:relative-coh-vanishing-Flag}
  Let $\bb{F}=\bb{F}_{\BP}(\mc{E})$ denote a flag bundle with structure morphism $\pi:\bb{F}\lra\BP$ as in Section~\ref{subsec:flag-theory}, and let $\rank(\mc{E})=n$. If $\ll=(\ll_1,\dots,\ll_{n-1})$ is a partition of size $|\ll|\leq (n-1-j)q-1$ and $e\geq(1+j)q$, then
 \[ R^k\pi_*\left(\mc{O}_{\bb{F}}(\ll_1,\dots,\ll_{n-1},e)\right) = 0 \quad\text{ for }k\leq j.\]
\end{theorem}

We will prove Theorem~\ref{thm:relative-coh-vanishing-Flag} as a consequence of the following related vanishing result.

\begin{theorem}\label{thm:relative-coh-vanishing-small-lam}
  Consider a flag bundle $\pi:\bb{F}\lra\BP$ as in Theorem~\ref{thm:relative-coh-vanishing-Flag}. Let $\ll=(\ll_1,\dots,\ll_{n-1})$ be a partition with $q-1\geq \ll_1$, if $j\leq n-2$ and $e\geq(1+j)q$, then
 \[ R^k\pi_*\left(\mc{O}_{\bb{F}}(\ll_1,\dots,\ll_{n-1},e)\right) = 0 \quad\text{ for }k\leq j.\]
\end{theorem}

\begin{proof}
We argue by induction on $n$. Using the factorization \eqref{eq:factorization-pi-fg} we identify $\bb{F} = \bb{F}_{\PP}(\mc{R})$ and consider the partition $\mu=(\ll_2,\dots,\ll_{n-1})$. If $j\leq n-3$, since $\ll_2\leq q-1$, by induction on $n$ we  have
 \[R^k f_*\left(\mc{O}_{\bb{F}}(\ll_2,\dots,\ll_{n-1},e)\right) = 0 \text{ for }k\leq j.\]
The desired vanishing then follows from the implication \eqref{eq:inductive-vanishing} with $y_i=\ll_i$ for $i\leq n-1$, and $y_n=e$.
We are left with the case when $j=n-2$, when the argument above can be applied to conclude
 \[R^k f_*\left(\mc{O}_{\bb{F}}(\ll_2,\dots,\ll_{n-1},e)\right) = 0 \text{ for }k\leq n-3.\]
Moreover, it follows from \cite{rai-vdb1}*{Theorem~3.1} that there exists a polynomial functor $\mc{P}$ of degree $d=\ll_2+\cdots+\ll_{n-1}+e$ such that
\[ R^{n-2} f_*\left(\mc{O}_{\bb{F}}(\ll_2,\dots,\ll_{n-1},e)\right) = \mc{P}(\mc{R}).\]
To prove that $R^{n-2}\pi_*\left(\mc{O}_{\bb{F}}(\ll_1,\dots,\ll_{n-1},e)\right) = 0$, by the projection formula it suffices to show
\[ g_*\left(\mc{Q}^{\ll_1}\oo \mc{P}(\mc{R})\right) = 0.\]
We may assume that $\mc{P}=\bb{L}_{\mu}$ is simple, for $\mu=(\mu_1,\dots,\mu_{n-1})$ a partition of $d$. Using the inclusion $\bb{L}_{\mu}\hookrightarrow\bb{S}_{\mu}$ and that $g_*$ is left exact, we also assume $\mc{P}=\bb{S}_{\mu}$. By \cite{BCRV}*{Theorem~9.8.5} we have
\[ \mc{P}(\mc{R}) = \bb{S}_{\mu}(\mc{R}) = f_*\left(\mc{O}_{\bb{F}}(\mu)\right)\]
and the projection formula implies
\begin{equation}\label{eq:gstr=pistr}
 g_*\left(\mc{Q}^{\ll_1}\oo \mc{P}(\mc{R})\right) = \pi_*\mc{O}_{\bb{F}}(\ll_1,\mu_1,\dots,\mu_{n-1}).
\end{equation}
Since $d\geq e \geq (n-1)q$, we have $\mu_1\geq q$. It follows that $\ll_1\leq q-1 < \mu_1$ and thus the weight $(\ll_1,\mu_1,\dots,\mu_{n-1})$ is not dominant. Using again \cite{BCRV}*{Theorem~9.8.5} we get the vanishing of the sheaves in \eqref{eq:gstr=pistr}, which concludes the proof.
\end{proof}

\begin{proof}[Proof of Theorem~\ref{thm:relative-coh-vanishing-Flag}]
 If $j\geq n-1$ then there is nothing to prove, because there is no partition of size $|\ll|\leq (n-1-j)q-1<0$. We may therefore assume that $0\leq j\leq n-2$.
 
 If $\ll_1\leq q-1$ then the desired conclusion follows from Theorem~\ref{thm:relative-coh-vanishing-small-lam}.
 Otherwise, we have $\ll_1\geq q$ and the hypothesis $|\ll|\leq (n-1-j)q-1$ forces $j\leq n-3$. Using the factorization \eqref{eq:factorization-pi-fg} we identify $\bb{F} = \bb{F}_{\PP}(\mc{R})$ and consider the partition $\mu=(\ll_2,\dots,\ll_{n-1})$. Since 
 \[ |\mu| = \ll_2+\cdots+\ll_{n-1} = |\ll| - \ll_1 \leq (n-2-j)q-1,\]
 we can conclude by induction on $n$ that
 \[R^k f_*\left(\mc{O}_{\bb{F}}(\ll_2,\dots,\ll_{n-1},e)\right) = 0 \text{ for }k\leq j.\]
We obtain the desired vanishing by the implication \eqref{eq:inductive-vanishing} with $y_i=\ll_i$ for $i\leq n-1$, and $y_n=e$.
\end{proof}

Our result on cohomology vanishing on flag varieties now follows as the special case of Theorem~\ref{thm:relative-coh-vanishing-Flag} where $\BP=\op{Spec}(\kk)$ and $\mc{E}=\kk^n$. It plays a key role in establishing the corresponding Cohomology Vanishing on Projective Space, which we explain next.

\begin{proof}[Proof of \eqref{eq:vanishing-Polyfun}] Consider first the special case when $\mc{P}=\bb{S}_{\ll}$ for some partition $\ll$. If $\ll_n>0$ then $\bb{S}_{\ll}\mc{Q}=0$, since $\mc{Q}$ has rank $n-1$. Otherwise we may write $\ll=(\ll_1,\dots,\ll_{n-1})$ and we have
\begin{equation}\label{eq:vanishing-coh-Schur}
  H^k\left(\PP,\bb{S}_{\ll}(\mc{Q})\oo \mc{O}_{\PP}(-e)\right) = H^k\left(Fl_n,\mc{O}_{Fl_n}(\ll_1,\dots,\ll_{n-1},e)\right) = 0 \quad\text{ for }k\leq j,
\end{equation}
where the first equality follows from \cite{BCRV}*{Theorem~9.8.5} and the projection formula, while the second equality comes from \eqref{eq:vanishing-flag-varieties}.

We prove \eqref{eq:vanishing-Polyfun} by induction on $k$, with the trivial base case $k=-1$. Without loss of generality it suffices to consider when $\mc{P}=\bb{L}_{\ll}$ is simple. Writing $\mc{P}'=\bb{S}_{\ll}/\bb{L}_{\ll}$ we get a short exact sequence
\begin{equation}\label{eq:ses-Llam-Slam-Ppr}
0\lra \bb{L}_{\ll}(\mc{Q})\oo \mc{O}_{\PP}(-e) \lra \bb{S}_{\ll}(\mc{Q})\oo \mc{O}_{\PP}(-e) \lra \mc{P}'(\mc{Q})\oo \mc{O}_{\PP}(-e) \lra 0.
\end{equation}
We assume that $k\leq j$ and, by induction on $k$, we assume that 
\[H^{k'}\left(\PP,\bb{L}_{\ll}(\mc{Q})\oo \mc{O}_{\PP}(-e)\right)  = H^{k'}\left(\PP,\mc{P}'(\mc{Q})\oo \mc{O}_{\PP}(-e)\right) = 0\text{ for }k'<k.\] 
If we consider the long exact sequence in cohomology associated to \eqref{eq:ses-Llam-Slam-Ppr}, using \eqref{eq:vanishing-coh-Schur} we obtain that $H^{k}\left(\PP,\bb{L}_{\ll}(\mc{Q})\oo \mc{O}_{\PP}(-e)\right)=0$, concluding the inductive step and our proof.
\end{proof}

\subsection{Cohomology as a polynomial representation}
\label{subsec:coh-polynomial-rep}

As illustrated in \cite{rai-vdb1}, the cohomology groups of equivariant vector bundles on projective space (and, more generally, on flag varieties) often carry the structure of polynomial representations. In Section~\ref{subsec:lowbd-symdet-coh}, we will employ the following enhanced polynomiality result.

\begin{theorem}\label{thm:poly-rep-coh}
 Let $\mc{Q}$ denote the tautological quotient bundle on the projective space $\PP=\PP^{n-1}$ as in \eqref{eq:taut-ses-PP}. If $\mc{P}$ is a polynomial functor of degree $d$, and $e\geq 1$, then for each $j\geq 0$ there exists a polynomial representation $W_j$ of degree $d+e-n$ such that
 \[ H^j\left(\PP,\mc{P}(\mc{Q}) \oo \mc{O}_{\PP}(e)\right) = L_{\omega_n}\oo W_j.\]
 In particular, if $d+e<n$ then the cohomology groups above vanish identically.
\end{theorem}

As it was the case for cohomology vanishing, in order to prove Theorem~\ref{thm:poly-rep-coh} it will be useful to establish a related statement on the full flag variety. This is obtained as a slight modification of \cite{rai-vdb1}*{Theorem~3.1}, and our proof follows closely the arguments in \cite{rai-vdb1}.

\begin{theorem}\label{thm:coh-lamn-poly}
 Suppose that $\ll\in\bb{Z}^n$ satisfies the conditions
 \begin{equation}\label{eq:lbd-lli}
  \ll_i > i-n \text{ for all }i=1,\dots,n.
 \end{equation}
 Then for each $j\geq 0$, there exists a polynomial representation $W_j$ of degree $|\ll|-j$ such that 
 \[H^j\left(Fl_n,\mc{O}_{Fl_n}(\ll)\right)=L_{\omega_n}\oo W_j.\] 
 In particular, if the conditions in \eqref{eq:lbd-lli} hold and $|\ll|<n$, then $H^j\left(Fl_n,\mc{O}_{Fl_n}(\ll)\right)=0$ for all $j$.
\end{theorem}

\begin{proof} To prove the theorem, we verify that if $L_{\mu}$ is a simple composition factor of $H^j\left(Fl_n,\mc{O}_{Fl_n}(\ll)\right)$, then $\mu_n\geq 1$. Suppose by contradiction that $\mu_n\leq 0$ and let $\chi\in\bb{Z}^n_{dom}-\rho$ and $w\in\mf{S}_n$ such that $\ll = w\bullet\chi$, where $\mf{S}_n$ denotes the symmetric group, and $\bullet$ denotes the dot action on weights given by $\sigma\bullet\ll = \sigma(\ll+\rho)-\rho$, where $\rho=(n-1,n-2,\dots,0)$. It follows that the sequences
\[(\ll_1-1,\dots,\ll_n-n)\quad\text{and}\quad(\chi_1-1,\dots,\chi_n-n)\]
agree up to a permutation of the entries. By \cite{And-str-link}*{Theorem~1}, we have that $\chi_n\leq\mu_n$. Since $\mu_n\leq 0$, it follows that $\chi_n-n\leq -n$ and therefore we must have 
\[\ll_i-i = \chi_n-n \leq -n\quad\text{ for some }i=1,\dots,n.\]
This contradicts \eqref{eq:lbd-lli} and concludes our proof of polynomiality. Since $H^j\left(Fl_n,\mc{O}_{Fl_n}(\ll)\right)$ is a representation of degree $|\ll|$ and $L_{\omega_n}$ has degree $n$, it follows that $W_j$ has degree $|\ll|-n$, so it can only be a non-zero polynomial representation if $|\ll|\geq n$.
\end{proof}

We can now verify Theorem~\ref{thm:poly-rep-coh}, using a strategy similar to the one used to establish Cohomology Vanishing on Projective Space in the previous section (see also the proof of \cite{rai-vdb1}*{Theorem~4.1}).

\begin{proof}[Proof of Theorem~\ref{thm:poly-rep-coh}] As in the proof of \eqref{eq:vanishing-Polyfun}, we consider first the special case when $\mc{P}=\bb{S}_{\ll}$ is a Schur functor, and we may assume that $\ll_n=0$ so that $\bb{S}_{\ll}(\mc{Q})\neq 0$. We have as in \eqref{eq:vanishing-coh-Schur} that
\[  H^j\left(\PP,\bb{S}_{\ll}(\mc{Q})\oo \mc{O}_{\PP}(-e)\right) = H^j\left(Fl_n,\mc{O}_{Fl_n}(\ll_1,\dots,\ll_{n-1},e)\right).\]
The desired conclusion for $\mc{P}=\bb{S}_{\ll}$ now follows from Theorem~\ref{thm:coh-lamn-poly}, where we set $\ll_n=e$: indeed, the hypothesis \eqref{eq:lbd-lli} applies since $\ll_i\geq 0 > i-n$ for $i\leq n-1$, and $\ll_n=e\geq 1 > n-n$.

To prove the theorem in general, we may without loss of generality assume that $\mc{P}=\bb{L}_{\ll}$ is a simple polynomial functor, and we argue by induction on partitions $\ll$, relative to the dominance order. We write $\mc{P}'=\bb{S}_{\ll}/\bb{L}_{\ll}$ and consider the associated short exact sequence \eqref{eq:ses-Llam-Slam-Ppr}. Using the previously established case of Schur functors, we write $H^j\left(\PP,\bb{S}_{\ll}(\mc{Q})\oo \mc{O}_{\PP}(-e)\right) = L_{\omega_n}\oo W'_j$ for some polynomial representations $W'_j$. Since $\mc{P}'$ has composition factors $\bb{L}_{\mu}$ with $\mu<\ll$, we obtain by induction that $H^j\left(\PP,\mc{P}'(\mc{Q})\oo \mc{O}_{\PP}(-e)\right) = L_{\omega_n}\oo W''_j$ for some polynomial representations $W''_j$. The conclusion for $\mc{P}=\bb{L}_{\ll}$ follows from the long exact sequence in cohomology associated to \eqref{eq:ses-Llam-Slam-Ppr}.
\end{proof}

\section{Generic determinantal rings}
\label{sec:Fthresh-detl}

The goal of this section is to verify part (1) of our Main Theorem, describing the diagonal $F$-threshold for the generic determinantal hypersurface. We in fact compute the diagonal $F$-threshold more generally, for the determinantal ring defined by the maximal minors of a generic $m\times n$ matrix.

\begin{theorem}\label{thm:Fthresh-detl}
Let $\kk$ denote a field of characteristic $p>0$, and let $S = \kk[x_{ij}]$ denote the polynomial algebra in the entries $(x_{ij})$ of the generic $m\times n$ matrix, where $m\geq n\geq 1$. If we let
\[ I_n = \langle n\times n \text{ minors of }(x_{ij})\rangle,\]
and set $R=S/I_n$ for the corresponding determinantal ring, then the $F$-threshold of $R$ is given by
\[ c(R) = m(n-1).\]
\end{theorem}

\begin{proof} We write $V=\kk^n$, $W=\kk^m$, and identify the polynomial ring $S=\kk[x_{ij}]$ with the symmetric algebra $\Sym(V\oo W)$. We set $q=p^s$ for some $s\geq 0$ and we study the cokernel of the map
\begin{equation}\label{eq:Fq-Rd-q-to-Rd}
 F^q(V\oo W) \oo R_{d-q} \lra R_d,
\end{equation}
in order to characterize when it is zero. The idea is to realize the map \eqref{eq:Fq-Rd-q-to-Rd} cohomologically. We write $\PP$ for the projective space parametrizing lines in $V$, and consider the tautological sequence
\begin{equation}\label{eq:taut-ses-PV}
0 \lra \mc{R} \lra V \oo \mc{O}_{\PP} \lra \mc{Q} \lra 0,
\end{equation}
where $\mc{Q}$ is the universal rank $n-1$ quotient bundle, and $\mc{R}=\mc{O}(-1)$ is the universal rank $1$ subbundle. If we tensor the tautological sequence by $W$ then we get
\begin{equation}\label{eq:taut-seq-matrices}
 0 \lra \xi \lra (V\oo W) \oo  \mc{O}_{\PP} \lra \eta \lra 0,
\end{equation}
where $\xi = \mc{R} \oo W \simeq \mc{O}_{\PP}(-1)^{\oplus m}$ and $\eta = \mc{Q} \oo W \simeq \mc{Q}^{\oplus m}$. We have the following natural maps
\[ F^q(V\oo W) \oo  \mc{O}_{\PP} \lra F^q\eta \subset \Sym^q \eta,\quad \Sym^q \eta \oo \Sym^{d-q}\eta \lra \Sym^d \eta.\]
Moreover, we have  that $R_d = H^0(\PP,\Sym^d\eta)$ (see for instance \cite{BCRV}*{Lemma~10.8.4}), and the map \eqref{eq:Fq-Rd-q-to-Rd} is the map induced on global sections by the composition $\psi_{q,d}$ of the natural morphisms
\[
\xymatrix{
 {F^q(V\oo W) \oo \Sym^{d-q}\eta} \ar[r] \ar@/^2.5pc/[rrr]^-{\psi_{q,d}} & F^q\eta \oo \Sym^{d-q}\eta  \ar[r] & \Sym^q\eta \oo \Sym^{d-q}\eta \ar[r] &  \Sym^d \eta  \\
 }
\]
The cokernel of $\psi_{q,d}$ is the truncated power $T_q\Sym^d\eta$, which is zero if 
\begin{equation}\label{eq:d-large} 
 d>(q-1)\cdot \rk(\eta) = (q-1)\cdot m\cdot(n-1).
\end{equation}
To prove the surjectivity of $H^0(\PP,\psi_{q,d})$, we consider the Koszul complex $\mc{K}^{\bullet}$ (we use cohomological notation and write $\mc{K}^{-i}$ instead of $\mc{K}_i$) with components
\[ \mc{K}^{-i} = \bw^i(F^q(V\oo W)) \oo \Sym^{d-iq}\eta\quad\text{ for }i\geq 0.\]
This convention allows us to extend the map $\psi_{q,d}:\mc{K}^{-1}\lra\mc{K}^0$.  The sheaves $\mc{K}^{-i}$ have vanishing higher cohomology (see for instance \cite{BCRV}*{Lemma~10.8.4}), hence the hypercohomology of $\mc{K}^{\bullet}$ is computed from the complex of global sections $H^0\left(\PP,\mc{K}^{\bullet}\right)$. In particular, we obtain that
\begin{equation}\label{eq:coker-psi=HH0} 
\coker H^0(\PP,\psi_{q,d}) = \bb{H}^0(\mc{K}^{\bullet}).
\end{equation}
Let $\mc{H}^{-i}$ be the cohomology of $\mc{K}^{\bullet}$ in degree $-i$, and consider the hypercohomology spectral sequence
\[E_2^{-i,j} = H^j\left(\PP,\mc{H}^{-i}\right) \Longrightarrow \bb{H}^{j-i}\left(\mc{K}^{\bullet}\right).\]
Condition \eqref{eq:d-large} implies that $E_2^{0,0}=0$, and therefore the vanishing $\bb{H}^0(\mc{K}^{\bullet})=0$ reduces to proving
\begin{equation}\label{eq:Hi-Hi=0} 
H^i(\PP,\mc{H}^{-i}) = 0\quad\text{ for }i\geq 1.
\end{equation}
Moreoer, using \eqref{eq:taut-seq-matrices} and the notation \eqref{eq:def-TqSymd} we obtain that for every $i$ we have
\[ \mc{H}^{-i} = F^q\left(\bw^i\xi\right) \oo T_q\Sym^{d-iq}\eta = \bigoplus (T_q\Sym^{d-iq}\eta)\oo\mc{O}_{\PP}(-iq) ,\]
where the right side has $m\choose i$ isomorphic summands, indexed by some choice of basis for $F^q\left(\bw^i W\right)$. 
If $N=(q-1)\cdot\rk(\eta) = (q-1)\cdot m\cdot(n-1)$, then we have a perfect pairing (see \cite{gao-rai}*{Section~2.2})
\[ T_q\Sym^{d-iq}\eta \times T_q\Sym^{N-d+iq}\eta \lra \det(\eta)^{q-1}\simeq\mc{O}_{\PP}((q-1)m).\]
By Serre duality we obtain
\begin{equation}\label{eq:Hi-of-Hi}
\begin{aligned}
 H^i\left(\PP,\mc{H}^{-i}\right)^{\vee} &= \bigoplus H^{n-1-i}\left(\PP,T_q\Sym^{N-d+iq}\eta\oo\mc{O}_{\PP}(iq-(q-1)m-n)\right) \\
 &= H^{n-1-i}\left(\PP,\mc{P}(\mc{Q})\oo \mc{O}_{\PP}(-e)\right),
 \end{aligned}
\end{equation}
where $\mc{P}$ is a polynomial functor of degree $N-d+iq<iq$ (using \eqref{eq:d-large}), and $e=(q-1)m+n-iq$. If we set $j=n-1-i$ then 
\[ e - (1+j) q = (q-1)\cdot m+n-iq - (n-i)\cdot q = (m-n)\cdot(q-1)\geq 0.\]
We can therefore apply \eqref{eq:vanishing-Polyfun} to deduce the vanishing of the cohomology groups in \eqref{eq:Hi-of-Hi}.
We conclude that the map \eqref{eq:Fq-Rd-q-to-Rd} is surjective for $d>N$, and therefore $v_R(q)\leq N$. Dividing by $q$ yields
\[ \frac{v_R(q)}{q} \leq \frac{(q-1)\cdot m\cdot(n-1)}{q}.\]
Taking the limit as $q\lra\infty$ gives the desired bound $c(R)\leq m\cdot(n-1)$ and the equality follows.
\end{proof}

\section{Generic symmetric determinant}
\label{sec:Fthresh-symm-det}

The goal of this section is to prove part (2) of our Main Theorem, computing the diagonal $F$-threshold for the determinant of a generic symmetric matrix. We prove the following.

\begin{theorem}\label{thm:Fthresh-symm-det}
Let $\kk$ denote a field of characteristic $p>0$, and let $S = \kk[x_{ij}]$ denote the polynomial algebra in the entries $(x_{ij})$ of the generic symmetric $n\times n$ matrix. If we let $f=\det(x_{ij})$ denote the symmetric determinant and $R=S/\langle f\rangle$, then the diagonal $F$-threshold of $R$ is given by
\[ c(R) = \frac{n^2-1}{2}.\]
\end{theorem}

To establish the upper bound $c(R)\leq (n^2-1)/2$ we employ the strategy from the case of general matrices: we first realize the graded components of $R$ cohomologically and then we obtain an upper bound for $v_R(q)$ based on cohomology vanishing results from Section~\ref{sec:coh-vanishing}. This is done in Section~\ref{subsec:upper-bound}. Following the discussion in Section~\ref{subsec:hyper}, we get $-a(R)={n+1\choose 2}-n = {n\choose 2}$, showing that the lower bound \eqref{eq:cR>aR-hypersurface} is no longer optimal. In Section~\ref{subsec:char2-symdet} we find an optimal lower bound for $v_R(q)$ in characteristic $2$ using duality theory. In characteristic $p>2$ we instead obtain the lower bound based on a representation theoretic argument and a careful analysis of cohomology in Section~\ref{subsec:lowbd-symdet-coh}.

\subsection{The upper bound}\label{subsec:upper-bound} We set $V=\kk^n$ and note that \eqref{eq:taut-ses-PV} yields a short exact sequence
\begin{equation}\label{eq:taut-seq-symm}
 0 \lra \xi \lra \Sym^2 V \oo  \mc{O}_{\PP} \lra \eta \lra 0,
\end{equation}
where $\xi = V\oo \mc{O}_{\PP}(-1)$ and $\eta = \Sym^2\mc{Q}$. Following \cite{weyman}*{Section~6.3}, we can describe the graded components of $R$ as $R_d=H^0(\PP,\Sym^d\eta)$. Hence, finding an upper bound $d>v_R(q)$ (for $q=p^s$) amounts to prove that the following natural multiplication map is surjective
\begin{equation}\label{eq:Fq-Rd-q-to-Rd-symm}
F^q(\Sym^2 V) \oo R_{d-q} \lra R_d.
\end{equation}
As for generic matrices, we realize \eqref{eq:Fq-Rd-q-to-Rd-symm} as the map induced on global sections by the composition
\[
\xymatrix{
 {F^q(\Sym^2 V) \oo \Sym^{d-q}\eta} \ar[r] \ar@/^2.5pc/[rrr]^-{\psi_{q,d}} & F^q\eta \oo \Sym^{d-q}\eta  \ar[r] & \Sym^q\eta \oo \Sym^{d-q}\eta \ar[r] &  \Sym^d \eta  \\
 }
\]
We note that
\[ \rank(\eta) = {n\choose 2}\quad\text{ and }\quad \det(\eta) = \mc{O}_{\PP}(n).\]
In particular we have  $T_q\Sym^d\eta = \coker(\psi_{q,d})= 0$ for $d>N = (q-1){n\choose 2}$. We consider the Koszul complex $\mc{K}^{\bullet}$ induced by the morphism $F^q(\Sym^2 V)\oo\mc{O}_{\PP}\lra \Sym^q\eta$, with components given by
\begin{equation}\label{eq:terms-kosz-sym} 
\mc{K}^{-i} = \bw^i(F^q(\Sym^2 V)) \oo \Sym^{d-iq}\eta\quad\text{ for }i\geq 0,
\end{equation}
and $\psi_{q,d}$ is the first differential $\mc{K}^{-1}\lra\mc{K}^0$. It follows from \cite{boffi-plethysm} that each $\Sym^m\eta$ has a filtration with composition factors $\bb{S}_{2\ll}\mc{Q}$ where $\ll$ is a partition of $m$, and each such factor $\bb{S}_{2\ll}\mc{Q}$ has vanishing higher cohomology (\cite{BCRV}*{Corollary~9.8.6}). It follows that each $\mc{K}^{-i}$ has vanishing higher cohomology, hence the analogue of \eqref{eq:coker-psi=HH0} holds. The homology sheaves of $\mc{K}^{\bullet}$ are computed~by
\[ \mc{H}^{-i} = F^q\left(\bw^i\xi\right) \oo T_q\Sym^{d-iq}\eta \simeq \bigoplus (T_q\Sym^{d-iq}\eta)\oo\mc{O}_{\PP}(-iq) \quad\text{ for all }i,\]
where the summands are indexed by a basis of $F^q\left(\bw^i V\right)$. Using Serre duality as in \eqref{eq:Hi-of-Hi}, we~obtain
\begin{equation}\label{eq:Hi-of-Hi-symm}
H^i\left(\PP,\mc{H}^{-i}\right)^{\vee} \simeq \bigoplus H^{n-1-i}\left(\PP,T_q\Sym^{N-d+iq}\eta\oo\mc{O}_{\PP}(iq-nq)\right)
= H^{n-1-i}\left(\PP,\mc{P}(\mc{Q})\oo \mc{O}_{\PP}(-e)\right),
\end{equation}
where $\mc{P}$ is a polynomial functor of degree $2(N-d+iq)$ and $e=(n-i)q$. If we take
\[ d = N + \frac{q(n-1)}{2}+1=\frac{q(n^2-1)}{2}-{n\choose 2} + 1,\]
then it follows that
\[ 2(N-d+iq) = (2i-n+1)q - 2 < iq\text{ for all }i\leq n-1.\]
If $j=n-1-i$, then $e=(1+j)q$, and we can apply \eqref{eq:vanishing-Polyfun} to deduce the vanishing of the cohomology groups in \eqref{eq:Hi-of-Hi-symm}. It follows that $\coker H^0(\PP,\psi_{q,d}) = \bb{H}^0(\mc{K}^{\bullet})=0$, hence \eqref{eq:Fq-Rd-q-to-Rd-symm} is surjective and
\begin{equation}\label{eq:upbd-vRq-symm}
v_R(q) \leq d-1 = q\cdot\frac{n^2-1}{2}-{n\choose 2}.
\end{equation}
Dividing by $q$ and taking the limit as $q\lra\infty$ we conclude that $c(R)\leq \frac{n^2-1}{2}$, as desired.

\subsection{The lower bound in characteristic $p=2$}\label{subsec:char2-symdet} 

We let $\mf{m}=\langle x_{ij}\rangle$ denote the maximal homogeneous ideal, and prove first that the symmetric determinant $f=\det(x_{ij})$ satisfies
\begin{equation}\label{eq:fx11-in-m2} 
f \equiv x_{11}\cdots x_{nn}\ (\text{mod }\mf{m}^{[2]})
\end{equation}
Indeed, we note that the symmetry $x_{ij}=x_{ji}$ implies that for every permutation $\sigma\in\mf{S}_n$ we have
\[ x_{1\sigma(1)}\cdots x_{i\sigma(i)}\cdots x_{n\sigma(n)} = x_{\sigma(1)1}\cdots x_{\sigma(i)i}\cdots x_{\sigma(n)n} = x_{1\sigma^{-1}(1)}\cdots x_{i\sigma^{-1}(i)}\cdots x_{n\sigma^{-1}(n)}. \]
Since $\op{char}(\kk)=2$, it follows that in the determinant formula
\[ f = \sum_{\sigma\in\mf{S}_n} x_{1\sigma(1)}\cdots x_{i\sigma(i)}\cdots x_{n\sigma(n)}\]
the terms corresponding to $\sigma$ and $\sigma^{-1}$ cancel out if $\sigma\neq\sigma^{-1}$. The remaining terms correspond to involutions in $\mf{S}_n$. If $\sigma$ is the identity, the corresponding term is $x_{11}\cdots x_{nn}$. Otherwise, there are $i\neq j$ such that $\sigma(i)=j$ and $\sigma(j)=i$, and the corresponding term is divisible by $x_{ij}^2$, proving \eqref{eq:fx11-in-m2}.

It follows from \eqref{eq:fx11-in-m2} that $f\cdot x_{11}\in \mf{m}^{[2]}$. This implies that for every $q=2^s$, $s\geq 1$, we have
\[ g= f^{q/2-1}\cdot x_{11}^{q/2} \in (\mf{m}^{[q]}:f).\]
Moreover, expanding $g$ in the monomial basis we obtain that $g\not\in \mf{m}^{[q]}$, as its support contains
\[  x_{11}^{q-1}x_{22}^{q/2-1} \cdots \ x_{nn}^{q/2-1} \not\in\mf{m}^{[q]}.\]
Using the notation in Section~\ref{subsec:hyper}, we consider $g$ as a non-zero element in $(0:_{\ol{S}} f)$, hence \eqref{eq:vRq-from-indeg0f} yields
\begin{equation}\label{eq:lowbd-vRq-symm-char2}
v_R(q) \geq (q-1){n+1\choose 2} - \deg(g) = (q-1){n+1\choose 2} - n(q/2-1) - q/2 = q\cdot\frac{n^2-1}{2} - {n\choose 2}.
\end{equation}
Dividing by $q$ and taking the limit as $q\lra\infty$ we obtain the desired lower bound $c(R)\geq \frac{n^2-1}{2}$.

\begin{remark}\label{rem:vRq-sharp-sym}
    Combining \eqref{eq:lowbd-vRq-symm-char2} with \eqref{eq:upbd-vRq-symm} shows that $v_R(q)=q\cdot\frac{n^2-1}{2}-{n\choose 2}$ in characteristic $p=2$.
\end{remark}

\subsection{The lower bound via cohomology and representation theory}\label{subsec:lowbd-symdet-coh}

We now assume $\op{char}(\kk)=p>2$, we let $q=p^s$ and set $d=(n^2-1)\cdot(q-1)/2$. Our goal is to prove $v_R(q)\geq d$. To that end, we form the Koszul complex \eqref{eq:terms-kosz-sym}, and consider the associated hypercohomology spectral sequence
\[E_2^{-i,j} = H^j\left(\PP,\mc{H}^{-i}\right) \Longrightarrow \bb{H}^{j-i}\left(\mc{K}^{\bullet}\right),\]
where the differential on page $r\geq 2$ is denoted $d_r^{-i,j} : E_r^{-i,j} \lra E_r^{-i-r+1,j+r}$. Using the analogue of \eqref{eq:coker-psi=HH0} in the case of symmetric matrices as in Section~\ref{subsec:upper-bound}, it follows that in order to prove that $v_R(q)\geq d$ we need to verify $\bb{H}^{0}\left(\mc{K}^{\bullet}\right)\neq 0$. This is achieved in the following steps:
\begin{enumerate}
    \item We identify an irreducible representation $L$ occurring in $E_2^{-(n-1),n-1}$.
    \item We show that $L$ does not occur in $E_2^{-i,i-1}$ for any $i\leq n-2$.
    \item The only potentially non-trivial differentials involving $E_r^{-(n-1),n-1}$ are maps with source $E_r^{-i,(i-1)}$ (when $r=n-i$), showing that $L$ occurs in $E_r^{-(n-1),n-1}$ for all $r$. In particular $E_{\infty}^{-(n-1),n-1}\neq 0$, and since it is a subquotient of $\bb{H}^{0}\left(\mc{K}^{\bullet}\right)$, we conclude that $\bb{H}^{0}\left(\mc{K}^{\bullet}\right)\neq 0$.
\end{enumerate}

\begin{lemma}\label{lem:detQ-in-TqSymeta}
 If $q=p^s>1$, $N=(q-1)\rk(\eta)$, and $d=(n^2-1)\cdot(q-1)/2$, then there exists a natural inclusion 
 \[(\det\mc{Q})^{\oo(q+1)}\hookrightarrow T_q\Sym^{N-d+(n-1)q}\eta.\]
\end{lemma}

\begin{proof} It follows from \cite{boffi-plethysm} that there exists a natural inclusion $\iota:(\det\mc{Q})^{\oo 2}=\bb{S}_{(2^{n-1})}\mc{Q}\hookrightarrow \Sym^{n-1}\eta$: in a local chart where $\mc{Q}$ is free of rank $(n-1)$ we can choose a corresponding local basis $\{y_{ij}:1\leq i\leq j\leq n-1\}$ for $\eta=\Sym^2\mc{Q}$, and then the image of $\iota$ is locally spanned by $\det(y_{ij})$. Raising the determinant to the power $(q+1)/2$ and considering the natural projection to the truncated powers we get a map
\begin{equation}\label{eq:det-embed-trunc}
(\det\mc{Q})^{\oo(q+1)}\lra \Sym^{(n-1)(q+1)/2}\eta \lra T_q\Sym^{(n-1)(q+1)/2}\eta.
\end{equation}
Since $(q+1)/2<q$, we have
\[\det(y_{ij})^{(q+1)/2} = y_{11}^{(q+1)/2}\cdots \ y_{n-1,n-1}^{(q+1)/2}  + \cdots \not\in\langle y_{ij}^q\rangle,\]
hence \eqref{eq:det-embed-trunc} is an inclusion. The conclusion follows from having $N-d+(n-1)q=(n-1)(q+1)/2$.
\end{proof}

To prove step (1), we first consider the natural identifications
\[
\begin{aligned}
E_2^{-(n-1),n-1} &= F^q\left(\bw^{n-1}V\right)\oo H^{n-1}\left(\PP,(T_q\Sym^{d-(n-1)q}\eta)\oo\mc{O}_{\PP}(-(n-1)q)\right) \\
&= F^q\left(\bw^{n-1}V\right)\oo H^0\left(\PP,(T_q\Sym^{N-d+(n-1)q}\eta)\oo\mc{O}_{\PP}(-q)\right)^{\vee} \oo (\det V)^{(q-1)n+1}.
\end{aligned}
\]
Lemma~\ref{lem:detQ-in-TqSymeta} implies that $H^0\left(\PP,(T_q\Sym^{N-d+(n-1)q}\eta)\oo\mc{O}_{\PP}(-q)\right)$ contains as a subrepresentation
\[H^0\left(\PP,(\det\mc{Q})^{\oo(q+1)}\oo\mc{O}_{\PP}(-q)\right)=H^0\left(Fl_n,\mc{O}_{Fl_n}(q+1,\dots,q+1,q)\right)=(\det V)^{q+1}\oo V^{\vee}\]
which follows from \cite{BCRV}*{Theorem~9.8.5}. This implies that $E_2^{-(n-1),n-1}$ contains
\begin{equation}\label{eq:def-L}
L = F^q\left(\bw^{n-1}V\right)\oo V \oo (\det V)^{qn-q-n}=L_{\mu},\text{ where }\mu = \omega_1+q\omega_{n-1}+(qn-q-n)\omega_n.
\end{equation}

We next verify step (2), which is equivalent to proving that for $i\leq n-2$, the representation
\[L^{\vee} = F^qV \oo \bw^{n-1}V \oo (\det V)^{-qn+n-1}\]
does not occur inside
\[
\begin{aligned}
\left(E_2^{-i,i-1}\right)^{\vee} &= F^q\left(\bw^iV\right)^{\vee}\oo H^i\left(\PP,(T_q\Sym^{d-iq}\eta)\oo\mc{O}_{\PP}(-iq)\right)^{\vee} \\
&=F^q\left(\bw^{n-i}V\right)\oo H^{n-1-i}\left(\PP,(T_q\Sym^{N-d+iq}\eta)\oo\mc{O}_{\PP}(-(n-i)q)\right)\oo (\det V)^{-qn+n-1-q}
\end{aligned}
\]
Since $(n-i)q\geq 1$, and $\det(V)=L_{\omega_n}$, it follows from Theorem~\ref{thm:poly-rep-coh} that
\[H^{n-1-i}\left(\PP,(T_q\Sym^{N-d+iq}\eta)\oo\mc{O}_{\PP}(-(n-i)q)\right) = W \oo \det(V)\]
for some polynomial representation $W$. Hence, in order to conclude step (2), we have to verify that
\[ F^qV \oo \bw^{n-1}V \oo (\det V)^{q-1} = F^q V \oo L_{\omega_{n-1}+(q-1)\omega_n}=L_{q\omega_1}\oo L_{\omega_{n-1}+(q-1)\omega_n}\]
does not occur in any tensor product $F^q\left(\bw^{n-i}V\right)\oo W=L_{q\omega_{n-i}}\oo W$ where $i\leq n-2$ and $W$ is a polynomial representation. Since $|\omega_{n-i}|=n-i>1=|\omega_1|$ and $\omega_{n-1}+(q-1)\omega_n$ is $q$-restricted, this follows from Lemma~\ref{lem:pres-in-tensor}.

Since step (3) follows from (1), (2), we conclude that $v_R(q)\geq d$. We divide by $q$ and take the limit as $q\lra\infty$ to obtain $c(R)\geq(n^2-1)/2$, which concludes the proof of Theorem~\ref{thm:Fthresh-symm-det}.

\section{Generic Pfaffian}
\label{sec:Fthresh-Pfaff}

The goal of this section is to give a quick derivation for the diagonal $F$-threshold for the Pfaffian of the generic skew-symmetric matrix of even size. We prove the following.

\begin{theorem}\label{thm:Fthresh-Pfaff}
Let $\kk$ be a field of characteristic $p>0$, and let $S = \kk[x_{ij}]$ be the polynomial algebra in the entries $(x_{ij})$ of a generic skew-symmetric $n\times n$ matrix, where $n=2m$ is even. Let $f=\op{Pf}(x_{ij})$ be the Pfaffian of $(x_{ij})$ and $R=S/\langle f\rangle$, then the diagonal $F$-threshold of $R$ is given by
\[ c(R) = \frac{n^2-2n}{2}.\]
\end{theorem}

\begin{proof}
To bound the diagonal $F$-threshold from below we use the general estimate $c(R)\geq -a(R)$: following the notation in Section~\ref{subsec:hyper}, we have $r=\dim(S)={n\choose 2}$ and $k=\deg(f)=n/2$, hence 
\[c(R) \overset{\eqref{eq:cR>aR-hypersurface}}{\geq} {n\choose 2}-\frac{n}{2} = \frac{n^2-2n}{2}.\]
To conclude, we then need to establish the upper bound
\begin{equation}\label{eq:upperbd-cR-pfaffian} 
c(R) \leq \frac{n^2-2n}{2} = 2(m^2-m).
\end{equation}
We do so by degenerating the hypersurface $f=\op{Pf}(x_{ij})$ to the determinant of a generic $m\times m$ matrix, and using the results from Section~\ref{sec:Fthresh-detl}. To that end, we define for $t\in\kk$
\[ x_{ij}(t) = \begin{cases}
    tx_{ij} & \text{if }1\leq i,j\leq m\text{ or }m+1\leq i,j\leq n, \\
    x_{ij} & \text{otherwise.}
\end{cases}\]
We let $f_t = \op{Pf}(x_{ij}(t))$ and $\mathstrut^{t}R = S/\langle f_t\rangle$. For $t\neq 0$ multiplication by $t$ is an invertible transformation, hence $\mathstrut^{t}R\simeq R$ as graded $\kk$-algebras,  $c(\mathstrut^{t}R)=c(R)$ and $v_q(\mathstrut^{t}R) = v_q(R)$ for $t\neq 0$ and $q=p^s$. 
Notice that for a given $d$, the condition $v_q(\mathstrut^{t}R)<d$ is equivalent to the surjectivity of the linear map
\[\left(S/\mf{m}^{[q]}\right)_{d-m} \overset{\cdot f_t}{\lra} \left(S/\mf{m}^{[q]}\right)_{d}\]
and, by semi-continuity, the rank of this transformation can only drop at $t=0$. This implies
\[ v_q(R) \leq v_q(\mathstrut^{0}R),\quad\text{and therefore}\quad c(R) \leq c(\mathstrut^{0}R).\]
Now it is enough to verify the bound \eqref{eq:upperbd-cR-pfaffian} holds for $c(\mathstrut^{0}R)$. The matrix $(x_{ij}(0))$ has a block structure
\[(x_{ij}(0)) = \begin{pmatrix}
0 & X \\
- X^T & 0
\end{pmatrix}\quad\text{where}\quad X = \begin{pmatrix}
x_{1,m+1} & x_{1,m+2} & \cdots & x_{1,n} \\
x_{2,m+1} & x_{2,m+2} & \cdots & x_{2,n} \\
\vdots & \vdots & \ddots & \vdots \\
x_{m,m+1} & x_{m,m+2} & \cdots & x_{m,n}
\end{pmatrix}.\]
We have that $X$ is a generic $m\times m$ matrix of variables, and that
\[ f_0 = \op{Pf}(x_{ij}(0)) = \det(X).\]
We can then partition the variables in $S$ accordingly to whether they appear in $x_{ij}(0)$ and get
\[ S = S_1 \oo_{\kk} S_2,\]
where $S_1 = \kk[x_{ij}:1\leq i,j\leq m\text{ or }m+1\leq i,j\leq 2m], \quad S_2 = \kk[x_{ij}:1\leq i\leq m\text{ and }m+1\leq j\leq n]$.
Since $f_0=\det(X)\in S_2$, it follows that 
\[ \mathstrut^{0}R = S_1 \oo_{\kk} R_2,\text{ where }R_2 = S_2 / \det(X).\]
The socle degree is additive with respect to tensor products over the base field $\kk$, hence it follows 
\[ v_q(\mathstrut^{0}R) = v_q(S_1) + v_q(R_2)\text{ and therefore }c(\mathstrut^{0}R)=c(S_1) + c(R_2).\]
Since $S_1$ is a polynomial ring of dimension $m(m-1)$, we get $c(S_1)=m(m-1)$, and since $R_2$ is a generic determinantal ring, it follows from Theorem~\ref{thm:Fthresh-detl} that $c(R_2)=m(m-1)$. We obtain that
\[c(R) \leq c(\mathstrut^{0}R) = c(S_1) + c(R_2) = 2(m^2-m),\]
which is the desired inequality \eqref{eq:upperbd-cR-pfaffian}, concluding the proof of Theorem~\ref{thm:Fthresh-Pfaff}.
\end{proof}

\section*{Acknowledgements}
We thank Francesco Russo and the organizers of the P.R.A.G.MAT.I.C. research school held in Catania in June 2023 for giving us the opportunity to meet. The authors also thank INDAM for their support. We thank
Luis Nuñez-Betancourt for suggesting the problem. 
Experiments with Macaulay2 \cite{GS} have provided many valuable insights. Raicu acknowledges the support of the National Science Foundation Grant DMS-2302341. Singh  acknowledges support of National Board for Higher Mathematics, DAE, Govt. of India (Grant no. 02011/24/2024/NBHM(R.P)/R\&D II/9826).

\end{document}